\documentclass[12pt]{amsart}

\usepackage{amsfonts,amsmath,latexsym,amssymb,verbatim,amsbsy,times}
\usepackage{amsthm, pdfsync}
\usepackage{pstricks}

\theoremstyle{plain}
\newtheorem{THEOREM}{Theorem}[section]

\newtheorem{theorem}[THEOREM]{Theorem}
\newtheorem{corollary}[THEOREM]{Corollary}
\newtheorem{lemma}[THEOREM]{Lemma}

\theoremstyle{definition}

\newtheorem{definition}[THEOREM]{Definition}

\theoremstyle{remark}

\newtheorem{claim}[THEOREM]{Claim}

\newcommand{\thm}[1]{Theorem~\ref{#1}}
\newcommand{\lem}[1]{Lemma~\ref{#1}}

\newcommand{\N}{\ensuremath{\mathbb{N}}}   
\newcommand{\R}{\ensuremath{\mathbb{R}}}   

\newcommand{\F}{\ensuremath{\mathbb{F}}}

\def \RR {\mathbb{R}}
\def \a {\alpha}

\def \e {\epsilon}
\def \f {\varphi}
\def \k {\kappa}
\def \l {\lambda}
\def \L {\Lambda}

\def \n {\nabla}

\def \w {\omega}

\def \cF {\mathcal{F}}

\def \< {\langle}
\def \> {\rangle}
\def \p {\partial}
\def \ra {\rightarrow}
\def \ss {\subset}

\begin{document}

\title[Regularity problems for the 3D NSE and Euler equations]
{A unified approach to regularity problems for the 3D Navier-Stokes and Euler
equations: the use of Kolmogorov's dissipation range}

\author{A. Cheskidov}
\thanks{The work of A. Cheskidov is partially supported by NSF grant DMS--0807827}
\address[A. Cheskidov]
{Department of Mathematics, Stat. and Comp. Sci.\\
       University of Illinois\\
       Chicago, IL 60607}
\email{acheskid@math.uic.edu}

\author{R. Shvydkoy}
\thanks{R. Shvydkoy acknowledges the support of NSF grant DMS--0907812}
\address[R. Shvydkoy]
{Department of Mathematics, Stat. and Comp. Sci.\\
       University of Illinois\\
       Chicago, IL 60607}
\email{shvydkoy@math.uic.edu}

\begin{abstract}
Motivated by Kolmogorov's theory of turbulence we present a unified
approach to the regularity problems for the 3D Navier-Stokes and Euler equations.
We introduce a dissipation wavenumber $\L(t)$ that separates
low modes where the Euler dynamics is predominant from the high
modes where the viscous forces take over. Then using
an indifferent to the viscosity technique we 
obtain a new regularity criterion which is 
weaker than every Ladyzhenskaya-Prodi-Serrin condition in the viscous case, and reduces to the Beale-Kato-Majda criterion in the inviscid case.
In the viscous case we also we prove that
Leray-Hopf solutions are regular provided  $\L \in L^{5/2}$, which
improves our previous $\L \in L^\infty$ condition. We also 
show that $\L \in L^1$ for all Leray-Hopf solutions.
Finally, we prove that Leray-Hopf solutions are regular when the time-averaged spatial intermittency is small, i.e.,
close to Kolmogorov's regime.
\end{abstract}

\maketitle

\section{Introduction}

We study the 3D incompressible fluid equations
\begin{equation} \label{NSE}
\left\{
\begin{aligned}
&\p_t u - \nu \Delta u + (u \cdot \nabla)u + \nabla p = 0, \qquad x \in \mathbb{R}^3, t > 0,\\
&\nabla \cdot u =0,\\
&u(0)=u_0,
\end{aligned}
\right.
\end{equation}
where $u(x,t)$, the velocity, and $p(x,t)$, the pressure, are unknowns,
$u_0 \in L^2(\mathbb{R}^3)$ is the initial condition,
and $\nu\geq0$ is the kinematic  viscosity coefficient of the fluid. In the inviscid ($\nu=0$) and viscous ($\nu>0$) cases the equations \eqref{NSE} are referred to as the Euler and Navier-Stokes (NSE)
equations respectively.

The regularity of solutions to the NSE remains a
significant open problem taking its mathematical roots in the seminal work of Leray \cite{Leray}. A considerable body of literature has been devoted to studying regularity criteria which include the classical Ladyzhenskaya-Prodi-Serrin conditions $u\in L_t^rL_x^s$, $2/r+3/s \leq 1$, $s>3$, its notable extention to the case $s=3$ by Escauriaza, Seregin, and \v{S}ver\'ak \cite{ESS}, logarithmic improvements of the above, conditions in terms of the pressure, velocity gradients, extentions to Besov-type spaces, etc. We refer the reader to \cite{amann,bv,ct,CS,veiga,jap,P} for detailed accounts. Although some subtle questions in the area remain open, e.g. regularity in the largest critical space $\dot B^{-1}_{\infty,\infty}$, it becomes increasingly convincing that the current techniques are not capable of narrowing the gap between the scaling invariant range $2/r+3/s = 1$ and the range of $2/r+3/s=3/2$ enjoyed by all Leray-Hopf solutions. 

The mathematical theory of \eqref{NSE} is even less complete in the inviscid case insofar as the existence of
weak solutions is not known. Nevertheless, the local existence and uniqueness of solutions to the Euler equations
was proved in $H^s$ for $s>5/2$ by Kato \cite{Kato}. Moreover, these solutions can be extended forward in time
as long as $\|\nabla \times u\|_\infty$ is integrable, which is known as the Beale-Kato-Majda condition \cite{BKM,jap,P}.
 
In the present paper we propose a unified approach to the regularity problem for the fluid equations \eqref{NSE} utilizing Kolmogorov's concept of an inertial range in turbulent flow. Our basic idea is to define a time-dependent dissipation wavenumber $\L(t)$ that separates high frequency modes where viscosity prevails over the non-linear term from the low frequency modes where the Euler dynamics is dominant. Specifically, we define
\[
\begin{split}
Q(t)&= \min\{q: 2^{-p}\|u_p(t)\|_\infty < c_0\nu, \ \forall \ p > q, q\geq 0\},\\
\L(t) &= 2^{Q(t)},
\end{split}
\]
where $c_0 >0$ is some absolute constant and $u_p$ denotes the Littlewood-Paley projection on the $p$-th dyadic shell (see Section \ref{s:nse}).

To illustrate the separatory role of the wavenumber $\L$ between viscous and inviscid properties of the equation we present a regularity criterion based on vorticity $\omega = \n \times u$, which reduces to the classical Beale-Kato-Majda condition  in the inviscid case. More precisely, using
an indifferent to the viscosity technique, we prove the following theorem:
\begin{theorem}
Let $u(t)$ be a weak solution of \eqref{NSE} with $\nu\geq0$, such that $u(t)$
is regular on $(0,T)$ and
\begin{equation} \label{BKM-intr}
\int_0^T \|\omega_{\leq Q(t)}(t)\|_{B^0_{\infty,\infty}} \, dt <\infty, 
\end{equation}
then $u(t)$ is regular on $(0,T]$.
\end{theorem}
Here the regularity should be understood as the continuity
of the $H^3$ norm.
If the viscosity coefficient $\nu$ is zero, then $\L(t) \equiv \infty$
and hence, the condition~\eqref{BKM-intr} naturally turns into the Beale-Kato-Majda criterion as stated in \cite{jap}. On the other hand,  we show that the condition~\eqref{BKM-intr} is weaker than every
Ladyzhenskaya-Prodi-Serrin condition in the viscous case. 

In the viscous case we also derive a regularity criterion in terms of only
$\L$. The estimates based on the use of $B^{-1}_{\infty,\infty}$-norm performed in
\cite{CS} allow us to conclude that $\L\in L^\infty(0,T)$ implies regularity of $u$ up to $T$. Here this condition will be weakened to $\L\in L^{5/2}(0,T)$, while for every Leray-Hopf solution one can show that $\L\in L^1(0,T)$.

With regard to connection between $\L$ and the classical Kolmogorov dissipation wavenumber we prove the direct inequality (up to an absolute multiple)
$$
\frac{1}{T} \int_0^T \L(t) dt  = \langle \L \rangle \lesssim \k_d = \left( \frac{\e}{\nu^3} \right)^{\frac{1}{1+d}},
$$
where $d$ is the intermittency parameter representing dimension of a dissipation set, and $\e = \nu \langle \| \n u\|_2^2 \rangle$ is the energy dissipation rate (see \cite{Frisch}). The definition of $d$ will be given in Section \ref{s:turb} based on the average level of saturation of Bernstein's inequalities inside the $\L(t)$-th dyadic shell. We then prove in \thm{t:inter} that if  $d>3/2$ then the solution $u$ is regular. This provides an analytical evidence to the fact that the flows usually observed in simulations of intermittent turbulence
are regular as they show only moderate deviations from the Kolmogorov's predicted value of $d =3$ (see \cite{Frisch} and references therein). 


Finally, in Section \ref{s:hyper} the use of $\L$ will be adapted to show regularity for a variant of
a hyperdissipative Navier-Stokes system and to give a short proof of a known recent result of Tao, \cite{T}.

\section{3D incompressible equations of fluid motion}\label{s:nse}
Let us first recall several classical definitions and results.
\begin{definition}
A weak solution of \eqref{NSE} with $\nu\geq 0$ on $[0,T]$ (or $[0,\infty)$ if $T=\infty$)
with the divergence-free initial data $u_0 \in L^2(\mathbb{R}^3)$ is a function
$u:[0,T] \to L^2(\RR^3)$ in the class
\[
u \in  C_{\mathrm{w}}([0,T];L^2(\RR^3)),
\]
satisfying $u(0)=u_0$,
\begin{multline}\label{weakform}
(u(t),\f(t))- (u_0,\f(0))\\ = \int_0^t \left\{  (u(s),\p_s \f(s)) + \nu (u(s),\Delta \f(s)) + (u(s)\cdot \n \f(s), u(s)) \right\} \, ds
\end{multline}
and $\n_x \cdot u(t) =0$ in the sense of distributions
for all $t\in[0,T]$ and all test functions
$\f \in C_0^\infty([0,T]\times \RR^3)$ with $\n_x \cdot \f = 0$. Here
$(\cdot,\cdot)$ stands for the $L^2$-inner product.
\end{definition}
In the case $\nu>0$ we also define Leray-Hopf solutions:
\begin{definition}
A weak solution $u(t)$ of \eqref{NSE} with $\nu>0$ on $[0,T]$ is called a
Leray-Hopf solution if $u \in L^2([0,T];H^1(\RR^3))$ and
the energy inequality
\begin{equation}\label{SEI}
\|u(t)\|_2^2 + 2\nu \int_{t_0}^t \|\n u(s)\|_2^2 \, ds \leq \|u(t_0)\|_2^2,
\end{equation}
is satisfied
for almost all $t_0 \in (0,T)$  and all $t \in (t_0,T]$.
\end{definition}
In 1934 Leray \cite{Leray} proved the existence of Leray-Hopf solutions
on $[0,\infty)$ for every $\nu>0$ and every divergence-free initial data
$u_0 \in L^2(\mathbb{R}^3)$. In the case $\nu=0$ the existence of weak solutions is still not known. Given $u_0 \in H^s$, the local existence
of $H^s$-continuous solutions is known for $s> 5/2$ in the
inviscid case, and for $s \geq 1/2$ in the viscous case.
Moreover, the continuity
of some supercritical $H^s$-norm ($s \geq 1/2$) of a Leray-Hopf solution implies the continuity of
all supercritical $H^s$-norms. Therefore, for simplicity, we choose the
continuity of the $H^3$-norm as a definition of regularity, even though 
any $s>5/2$ would be appropriate as well.

\begin{definition}
A week solution $u(t)$ of \eqref{NSE} with $\nu\geq 0$ is regular on a time
interval $\mathcal{I}$ if $\|u(t)\|_{H^{3}}$ is continuous on $\mathcal{I}$.
\end{definition}
An interval $\mathcal{I}$ where $u(t)$ is regular is called an interval of regularity.
Viscous regular solutions are unique in the class of Leray-Hopf solutions,
which results in the following regularity condition:

\begin{theorem}[Leray] \label{t:LHreg}
Let $u(t)$ be a Leray-Hopf solution of \eqref{NSE} with $\nu>0$ on $[0,T]$. If
for every interval of regularity $(\alpha,\beta) \subset (0,T)$
\[
\limsup_{t \to \beta -} \|u(t)\|_{H^{s}} < \infty,
\]
for some $s > 1/2$, then $u(t)$ is regular on $(0,T]$.
\end{theorem}

Note that merely weak viscous solutions are not known to satisfy the above property.
The possibility of a blow up from the right for such solutions
has not been ruled out.


Let us now briefly recall the definition of Besov spaces.
We will use the notation $\lambda_q = 2^q$ (in some inverse length
units). Let $B_r$ denote the ball centered at $0$ of radius $r$
in $\RR^{3}$. Let us fix a nonnegative radial function $\chi \in {C_0^{\infty}} (B_1)$ such that $\chi(\xi)=1$ for $|\xi|\leq
1/2$. We further define
$\f(\xi) = \chi(\l_1^{-1}\xi) - \chi(\xi)$ and $\f_q(\xi) = \f(\l_q^{-1}\xi)$.
For a tempered distribution vector field $u$ let us denote
\[
u_q = \cF^{-1}(\f_q) \ast u, \quad \text{ for } q > -1, \qquad 
u_{-1} = \cF^{-1}(\chi) \ast u,
\]
where $\cF$ denotes the Fourier transform. So, we have $u = \sum_{q=-1}^\infty u_q$
in the sense of distributions.
We also use the following notation
\[
u_{\leq Q} = \sum_{q \leq Q} u_q, \qquad u_{\geq Q} = \sum_{q \geq Q} u_q.
\]
Finally, let us recall that a tempered distribution $u$ belongs to
$B^s_{p,\infty}$ iff
$$
\|u\|_{B^s_{p,\infty}} = \sup_q \l_q^s \|u_q\|_p  < \infty.
$$

\section{A universal regularity criterion}
Let $u(t)$ be a weak solution of \eqref{NSE} with $\nu \geq 0$ on $[0,T]$.
We define our dissipation wavenumber as
\begin{equation}\label{lambda}
\L(t)=
\min\{\l_q: \l_p^{-1}\|u_p(t)\|_\infty < c_0\nu, \ \forall \ p > q, q\geq 0\},
\end{equation}
where $c_0$ is an absolute constant, which will be defined later. Note that
$\L(t) \equiv \infty$ in the inviscid case $\nu=0$.
Let $Q(t)\in \N$ be such that $\l_{Q(t)} =\L(t)$. Directly from the definition we have
\begin{equation}\label{reverseQ}
\|u(t)_{Q(t)}\|_\infty \geq c_0 \nu \L(t),
\end{equation}
provided $1<\L(t)<\infty$.

We will now obtain a universal regularity criterion, which is valid in both
viscous and inviscid cases. This criterion is stated
in terms of a Besov type bound on frequencies smaller than $\L(t)$.
For this purpose we consider the following function
\[
f(t)=\|u_{\leq Q(t)}(t)\|_{B^{1}_{\infty,\infty}}  =\sup_{q \leq Q(t)} \l_q\|u_q(t)\|_\infty,
\]
or in terms of vorticity $\w = \n \times u$,
\[
f(t) \sim \| \w_{\leq Q(t)}(t) \|_{B^0_{\infty,\infty}}.
\]
\begin{theorem} \label{thm:main}
Let $u$ be a weak solution to \eqref{NSE} with $\nu \geq 0$ on $[0,T]$.
Assume that $u(t)$ is regular on $[0,T)$, and $f \in L^1(0,T)$,  i.e.
\begin{equation} \label{BKM}
\int_0^T \|\omega_{\leq Q(t)}(t)\|_{B^0_{\infty,\infty}} \, dt <\infty. 
\end{equation}
Then $u(t)$ is regular on $[0, T]$. 
\end{theorem}
We note that the theorem is valid for both Navier-Stokes and the Euler equations. In the case when $\nu = 0$ we have $\L \equiv \infty$, and thus we recover the classical Beal-Kato-Majda criterion  extended to Besov spaces as in Kozono, Ogawa, and Taniuchi \cite{jap}, see also Planchon \cite{P}. Our technique of proving \thm{thm:main} does not distinguish between the cases $\nu>0$ and $\nu=0$ and is based on the frequency separation method and logarithmic Sobolev inequality.
\begin{proof}
Since $u(t)$ is regular on $[0,T)$, we have that $\|u(t)\|_{H^s}$
is continuous on $[0,T)$ for $s=3$. In what follows
we will prove the following estimate:
\begin{equation} \label{main-ineq}
\frac{1}{2}\frac{d}{dt}\|u(t)\|^2_{H^s} \leq C (1+f(t))\|u(t)\|^2_{H^s}
(1+\log_+\|u(t)\|_{H^s}),
\end{equation}
where $C$ is an absolute constant. Then the Gr\"{o}nwall inequality
implies that $\|u(t)\|_{H^s}$ is bounded on $[0,T)$. The regularity of $u$ now follows from the standard argument. Indeed, let $M=\sup_{t\in[0,T)}\|u(t)\|_{H^s}$. By the local well-posedness of regular solutions for $\e<\frac{1}{cM}$ there exists a regular solution $v \in C([T-\e,T];H^s)$ with $v(T-\e) = u(T-\e)$. By uniqueness, $v = u$ on $[T-\e,T)$. By weak continuity of $u$ we deduce that $u(T) = v(T)$.  Hence, $u(t)$ is regular on the closed interval $[0,T]$.

We now proceed to proving \eqref{main-ineq}.
Testing  the equation \eqref{NSE} against $\p^{2\a} u$ on the interval $[0,T)$, where $\a$ is a multiindex with $|\a|\leq s$, we obtain
\begin{equation*}
\frac{1}{2}\frac{d}{dt} \sum_{|\a| \leq s} \|\p^{\a} u\|_2^2 \leq  -\nu \|u\|_{H^{s+1}}^2 + \nu \|u\|_2^2+ 
 \sum_{\substack{|\a|\leq s\\ i \leq \a}} \left| \int \p^{\a-i} u \cdot \n \p^i u \cdot \p^\a u \right|.
\end{equation*}
Let us note that if $\nu>0$, then $u\in H^s$ implies $u\in H^{s+1}$ automatically on $[0,T)$ by the classical bootstrapping argument. The term $\nu \|u\|_2^2$  above was added to compensate the homogeneity of the viscous term near low frequencies. Furthermore, let us note that if $i = \a$ or if $\a=0$, the trilinear term vanishes due to incompressibility rendering the estimate
\begin{equation}\label{test-s}
\frac{1}{2}\frac{d}{dt} \|u\|_{H^s} \leq  -\nu \|u\|_{H^{s+1}}^2 + \nu \|u\|_2^2+
\sum_{\substack{1\leq |\a|\leq s\\ i < \a}} \left|\int \p^{\a-i} u \cdot \n \p^i u \cdot \p^\a u\right|.
\end{equation}
Now all the trilinear terms have at least one and at most $|\a|$ derivatives on each component.
We now prove the following auxiliary lemma.
\begin{lemma}\label{l:aux}
Suppose $1 \leq |\a| \leq s$, and $|\a_1| + |\a_2| + |\a_3| = 2|\a|+1$ and all indices satisfy $1\leq |\a_i| \leq |\a|$. Let $Q \in \N$ and let $u\in H^{s}$. Then
\begin{equation}\label{ }
\left|\int \p^{\a_1} u_{\leq Q} \p^{\a_2} u \p^{\a_3} u\right|  \leq C \|u_{\leq Q}\|_{B^{1}_{\infty,\infty}}\| u\|_{H^s} (1+ \log_+\|u\|_{H^s}).
\end{equation}  
\end{lemma}
\begin{proof}
We have 
\[
\begin{split}
\int \p^{\a_1} u_{\leq Q} \p^{\a_2} u \p^{\a_3} u &= \int \p^{\a_1} u_{\leq Q} \p^{\a_2} u_{\leq Q} \p^{\a_3} u_{\leq Q} \\ &+ \int \p^{\a_1} u_{\leq Q} \sum_{\substack{q' >Q, q'' >Q-2\\ |q'-q''|\leq 2}}\p^{\a_2} u_{q'} \p^{\a_3} u_{q''} = I + II
\end{split}
\]
To estimate $II$ we argue as follows
$$
| II | \lesssim \|\n u_{\leq Q} \|_\infty \l_Q^{\a_1 - 1} \sum_{q >Q-2}\l_q^{\a_2+\a_3} \|u_{q}\|_2^2.
$$
Since $s>3/2$, by the Besov-type logarithmic Sobolev inequality we have 
$$
\| \n u_{\leq Q} \|_\infty \lesssim   \|u_{\leq Q}\|_{B^{1}_{\infty,\infty}} (1+ \log_+\|u\|_{H^s})
$$ 
(see Appendix or \cite{jap} for more general versions). Thus, continuing the above,
\[
\begin{split}
| II | & \lesssim \|u_{\leq Q}\|_{B^{1}_{\infty,\infty}} (1+ \log_+\|u\|_{H^s})  \sum_{q >Q-2}\l_q^{\a_1 + \a_2+\a_3 - 1} \|u_{q}\|_2^2 \\
&\lesssim \|u_{\leq Q}\|_{B^{1}_{\infty,\infty}} (1+ \log_+\|u\|_{H^s}) \|u\|^2_{H^s}.
\end{split}
\]
As to $I$ we have the decomposition
\[
\begin{split}
I &= \int \sum_{q \leq Q, |q'-q|, |q''-q| \leq 2 }\p^{\a_1} u_q \p^{\a_2} u_{q'} \p^{\a_3} u_{\leq q''} \\
& + \int \sum_{q \leq Q, |q'-q|, |q''-q| \leq 2 }\p^{\a_1} u_q \p^{\a_2} u_{\leq q'} \p^{\a_3} u_{q''} \\
&+ \int \sum_{q \leq Q, |q'-q''| \leq 2, q' \geq q }\p^{\a_1} u_q \p^{\a_2} u_{q'} \p^{\a_3} u_{ q''}  - \text{repeated terms}
\end{split}
\]
Estimates for all of these terms are similar. We show one for the first term only. We have
\[
\begin{split}
\left| \int  \sum_{\substack{q \leq Q \\ |q'-q|\leq 2\\ |q''-q| \leq 2} }\p^{\a_1} u_q \p^{\a_2} u_{q'} \p^{\a_3} u_{\leq q''} \right| &
\lesssim  \sum_{\substack{q \leq Q \\ |q'-q|\leq 2\\ |q''-q| \leq 2} }\l_q^{\a_1+\a_2+\a_3-1}\| u_q \|_2 \| u_{q'} \|_2 \|\n u_{\leq q''} \|_\infty \\
 &\lesssim  \|u_{\leq Q}\|_{B^{1}_{\infty,\infty}} (1+ \log_+\|u\|_{H^s}) \sum_{q\leq Q} \l_q^{2s} \|u_q\|_2^2 \\ 
 &\lesssim  \|u_{\leq Q}\|_{B^{1}_{\infty,\infty}} (1+ \log_+\|u\|_{H^s})\|u\|_{H^s}^2.
\end{split}
\]
\end{proof}
Note that if $\L\equiv \infty$, then \lem{l:aux} already finishes the proof of
\eqref{main-ineq}. Otherwise, going back to \eqref{test-s} and in view of \lem{l:aux} the only terms left to estimate are
\[
\int \p^{\a-i} u_{>Q(t)} \cdot \n \p^i u_{>Q(t)} \cdot \p^{\a} u_{>Q(t)}.
\]
These will be absorbed by the viscous term after a proper frequency localization. We split each of them into the sum
\[
\begin{split}
\int \p^{\a-i} u_{>Q} \cdot \n \p^i u_{>Q} \cdot \p^\a u_{>Q} & = \int \sum_{\substack{q >Q\\ |q'-q|,|q''-q| \leq 2}} \p^{\a_1} u_{q'} \p^{\a_2} u_{q''} \p^{\a_3} u_{Q< \cdot \leq q} \\
&+ \text{cyclic terms } - \text{repeated terms}.
\end{split}
\]
We estimate
\[
\begin{split}
&\left| \int  \sum_{\substack{q >Q\\ |q'-q|,|q''-q| \leq 2}} \p^{\a_1} u_{q'} \p^{\a_2} u_{q''} \p^{\a_3} u_{Q< \cdot \leq q} \right| \\
&\qquad \lesssim \sum_{\substack{q >Q\\ |q'-q|,|q''-q| \leq 2}} \l_q^{\a_1+\a_2} \|u_{q'}\|_2 \|u_{q''}\|_2 \sum_{Q <p\leq q} \l_p^{\a_3+1} \l_p^{-1} \| u_p\|_\infty  \\
&\qquad  \lesssim  c_0 \nu \sum_{\substack{q >Q\\ |q'-q|,|q''-q| \leq 2}} \l_q^{\a_1+\a_2+\a_3+1} \|u_{q'}\|_2 \|u_{q''}\|_2\\
&\qquad  \lesssim c_0 \nu \|u\|_{ H^{s+1}}^2.
\end{split}
\]
By choosing $c_0$ small enough in the definition of $\L$ we then ensure that this term is smaller than the viscous term in \eqref{test-s}. This finishes the proof.
\end{proof}

Note that in the inviscid case $\nu=0$ condition \eqref{BKM} is a variant of the Beale-Kato-Majda condition.
In the viscous case $\nu>0$ this theorem together with Theorem~\ref{t:LHreg}
implies the following

\begin{corollary}
Let $u$ be a Leray-Hopf solution to \eqref{NSE} with $\nu > 0$  on $[0,T]$, such that
\begin{equation} \label{BKM-cor}
\int_0^T \|\omega_{\leq Q(t)}(t)\|_{B^0_{\infty,\infty}} \, dt <\infty.
\end{equation}
Then $u(t)$ is regular on $(0,T]$. 
\end{corollary}

In the next section we will further analyze the criterion \eqref{BKM-cor}
in the viscous case and show that it is weaker than every
Ladyzhenskaya-Prodi-Serrin condition.

\section{Regularity criteria in the viscous case}

In this section we will focus on the 3D Navier-Stokes equations, i.e.,
equations \eqref{NSE} with $\nu>0$.
In  \cite{CS} it was proved that a Leray-Hopf solution $u(t)$ is regular on $(0,T]$ provided
\begin{equation} \label{1}
\limsup_{q \to \infty} \sup_{t \in (0,T)} \l_q^{-1}
\|u_q(t)\|_{\infty} < c_0 \nu.
\end{equation}
In terms of $\L(t)$ this regularity condition  can be restated as
$\L \in L^\infty(0,T)$. 
It is also shown to be equivalent to the following small-jump condition:
\begin{equation} \label{jump-cond}
\sup_{t\in(0,T]}\limsup_{t_0 \to t-} \|u(t) - u(t_0)\|_{B^{-1}_{\infty,\infty}} < c_1 \nu.
\end{equation}
More precisely, \eqref{1} implies \eqref{jump-cond} with $c_1=2c_0$, and \eqref{jump-cond} implies \eqref{1} with $c_0 = 2c_1$. 

It is easy to see that Theorem~\ref{thm:main} improves this condition to
$\L \in L^{5/2}(0,T)$. Indeed, in view of \eqref{reverseQ} we have $f(t) \gtrsim \L(t)^2$, provided
$\L(t)>1$. On the other hand, by Bernstein inequality, 
\[
f(t) \lesssim  \sup_{q \leq Q(t)} \l^{5/2}_q\|u_q(t)\|_2 \lesssim 
\L(t)^{5/2}.
\]
In summary,
\begin{equation}\label{fL}
\L(t)^2 \lesssim f(t) \lesssim \L(t)^{5/2}, \qquad \text{whenever} \qquad
\L(t)>1.
\end{equation}
In particular, if $\L \in L^{5/2}(0,T)$ for some Leray-Hopf solution $u(t)$,
then $f \in L^1(0,T)$ and, consequently, $u(t)$ is regular on $(0,T]$ by
Theorem~\ref{thm:main}.
On the other hand, as we will see next, $\L(t)$ is integrable for every
Leray-Hopf solution.
Hence, this approach would require to fill the gap between $L^{1}$ and
$L^{5/2}$ in order to solve the regularity problem.

In what follows we will often integrate over the set $U=\{t: \L(t)>1\}$.
Note that $\L(t) < \infty$ a.e. for every Leray-Hopf solution.
Thus \eqref{reverseQ} implies
\begin{equation} \label{reverseQ-1}
\|u(t)_{Q(t)}\|_\infty \geq c_0 \nu \L(t), \qquad \text{for a.a.} \qquad  t \in U. 
\end{equation}

\begin{lemma}
Let $u(t)$ be a Leray-Hopf solution to \eqref{NSE} with $\nu>0$ on
$[0,T]$. Then  $\L \in L^1(0,T)$.
\end{lemma}
\begin{proof}
Let $u(t)$ be a Leray-Hopf solution on $[0,T]$ and $U = [0,T] \cup \{t: \L(t) > 1\}$. By Bernstein inequality
$\|u_p\|_\infty \leq \l_p^{3/2} \|u_p\|_2$. Therefore, since
$\|\nabla u(t)\|_2^2$ is integrable, using \eqref{reverseQ-1} we obtain
\[
c_0\nu \int_U \L(t) \, dt \leq 
\int_U \L(t)^{-1} \|u_{Q(t)}\|^2_\infty \, dt \leq 
\int_U \L(t)^2 \|u_{Q(t)}\|^2_2 \, dt  < \infty.
\]
\end{proof}

Now we will show that the viscous Beale-Kato-Majda condition $f \in L^1$ is
weaker than every Ladyzhenskaya-Prodi-Serrin regularity condition. First, observe that if $3/s+2/r = 1$, then
$$
L^rL^s \ss L^r B^{0}_{s,\infty} \ss L^r B^{2/r-1}_{\infty,\infty}.
$$
Hence, the latter is a larger class for regularity. We then have the following implication.
\begin{lemma} Let $u(t)$ be a Leray-Hopf solution to \eqref{NSE}
with $\nu>0$ on $[0,T]$.
If $u\in L^r((0,T);B^{2/r-1}_{\infty,\infty})$ for some $1\leq r < \infty$, then
$f \in L^1(0,T)$.
\end{lemma}
\begin{proof}
The case $r=1$ follows from the definition of $f$. Now assume that $r>1$.
Let $U = [0,T] \cap \{t: \L(t) > 1\}$.
Clearly,
\[
\int_{[0,T]\setminus U} f(t) \, dt \lesssim \int_0^T \|u(t)\|_2 \, dt < \infty.
\]
On the other hand, using \eqref{reverseQ-1} we obtain
\[
\begin{split}
\int_U f(t) \, dt &\leq \int_U \L(t)^{2-2/r} \sup_q \l_q^{2/r-1}\|u_q(t)\|_\infty \, dt\\
&\leq
 \left(\int_U \L(t)^2 \, dt\right)^{1-1/r} \left(\int_U \sup_q \l_q^{2-r}\|u_q(t)\|_\infty^r
\, dt\right)^{1/r}\\
&\leq c(\nu)
\left(\int_U \L(t)^{2-r} \|u_{Q(t)}(t)\|_\infty^r \, dt\right)^{1-1/r} \|u\|_{L^r B^{2/r-1}_{\infty,\infty}}\\
&\leq c(\nu)  \|u\|^{r-1}_{L^r B^{2/r-1}_{\infty,\infty}}  \|u\|_{L^r B^{2/r-1}_{\infty,\infty}}\\
&= c(\nu)  \|u\|^r_{L^r B^{2/r-1}_{\infty,\infty}},
\end{split}
\] 
which concludes the proof.
\end{proof}

The regularity criterion in Theorem~\ref{thm:main} also allows to obtain
conditions under which a ``mild'' blow-up
(blow-up in norms only higher than some $B^{-r}_{\infty,\infty}$) is possible.
By the energy inequality and Bernstein estimates, every Leray-Hopf solution 
satisfies $u \in L^\infty((0,T);B^{-3/2}_{\infty,\infty})$.
Now assume that $u$ belongs to a smoother space $L^\infty((0,T);B^{-r}_{\infty,\infty})$
for some $r \in [0,3/2]$. Then the complementary condition for regularity is
$\L \in L^{1+r}$.
\begin{corollary} \label{cor-reg}
If $u(t)$ is a Leray-Hopf solution of \eqref{NSE} on $[0,T]$, such that
$u \in L^\infty B^{-r}_{\infty,\infty}$ and $\L \in L^{1+r}$, then
$u(t)$ is regular on $(0,T]$.
\end{corollary}
\begin{proof}
Observe that
\[
f(t) = \sup_{q \leq Q(t)}\l_q^{-r} \|u_q(t)\|_\infty \l_q^{1+r}
\leq \|u\|_{L^\infty B^{-r}_{\infty,\infty}} \L(t)^{1+r}.
\]
Now the result follows from Theorem~\ref{thm:main}.
\end{proof}
In the scaling invariant case of $u \in L^\infty B^{-1}_{\infty,\infty}$ we find that $\L\in L^{2}$ is sufficient for regularity, although it is not clear whether this latter condition is already a consequence of the former.


\section{Connection with Kolmogorov's theory of turbulence}\label{s:turb}
Kolmogorov in 1941 \cite{K41} suggested that viscous effects in a turbulent flow are negligible in the inertial range,
the range below Kolmogorov's dissipation wave number
\[
\kappa_\mathrm{d} = \left(\frac{\epsilon}{\nu^3}\right)^{1/4},
\]
where the energy dissipation rate $\epsilon$ is usually defined as
\[
\epsilon = \nu\langle \|\nabla u\|_2^2\rangle =\frac{\nu}{T}\int_0^T \|\nabla u\|_2^2 \, dt.
\]
However, numerical and emperical observations show than one also needs to take into account effects of spacial intermittency which change the dimension of $\e$. With this intermittency correction the formula for $\kappa_\mathrm{d}$ becomes
\begin{equation} \label{defk_d}
\kappa_\mathrm{d} = \left( \frac{\e}{\nu^3} \right)^{\frac{1}{d+1}},
\end{equation}
where the parameter $d\in[0,3]$ represents the dimension
of the set in the physical space where the dissipation occurs. Then $s=3-d$
represents the dimension of the set in the Fourier space where dissipation occurs. We define $s$ so that it
captures the level of saturation of the Bernstein inequality in the last dyadic shell in the inertial range. More precisely, by Bernstein we have inequalities
$$
\langle \L^{-1}\| u_{Q}\|_\infty^2\rangle \lesssim \langle\L^2 \|u_Q\|_2^2 \rangle  \lesssim \langle\L^2 \|u_Q\|_\infty^2 \rangle.
$$
We let $s\in [0,3]$ be any parameter for which 
\begin{equation}\label{}
 \langle \L^{2-s}\| u_{Q}\|_\infty^2\rangle \lesssim \langle\L^2 \|u_Q\|_2^2 \rangle\ .
\end{equation}
Let $U = [0,T] \cap \{t: \L(t) > 1\}$
and
\[
\langle \L \rangle_U = \frac{1}{T}\int_U \L(t) \, dt.
\]
Then using \eqref{reverseQ-1}  we have
\[
\begin{split}
\langle \L\rangle - 1 &\leq  \langle \L \rangle_U
= \left(\frac{\langle \L(\nu c_0)^{\frac{2}{4-s}}\rangle_U^{4-s}}{\nu^{2}c_0^{2}} \right)^{\frac{1}{4-s}}
\leq \left(\frac{\langle \L^{\frac{2-s}{4-s}} \|u_Q\|_\infty^{\frac{2}{4-s}} \rangle^{4-s}}{\nu^{2}c_0^{2}} \right)^{\frac{1}{4-s}}\\
&\leq \left(\frac{\langle \L^{2-s} \|u_Q\|_\infty^2 \rangle}{\nu^{2}c_0^{2}} \right)^{\frac{1}{4-s}}
\lesssim \left(\frac{\langle \L^2 \|u_Q\|^2_2\rangle}{\nu^{2}c_0^{2}} \right)^{\frac{1}{4-s}}\\
&\lesssim \left(\frac{\epsilon}{\nu^{3}} \right)^{\frac{1}{4-s}}.
\end{split}
\]
So, we obtain
$$
\langle \L\rangle \lesssim \kappa_\mathrm{d}.
$$ 
Hence $\langle \L \rangle$ defines the dissipation range more precisely than $\kappa_\mathrm{d}$. In other words, generally the effects of viscosity start to manifest earlier than predicted by the dimensional argument.  Now we will show that Leray-Hopf solutions are regular
when $d = 3-s$ is close to $3$ and in particular in Kolmogorov 41 regime of $d=3$. 
\begin{theorem}\label{t:inter}
Let $u(t)$ be a Leray-Hopf solution of \eqref{NSE} on $[0,T]$ for which
$d> 3/2$. Then $u(t)$ is regular on $(0,T]$.
\end{theorem}
\begin{proof}
Let $U = [0,T] \cap \{t: \L(t) > 1\}$. Since $s<3/2$ we have via \eqref{reverseQ-1} 
\begin{equation}
\begin{split}
c_0^2\nu^2 \int_U \L(t)^{5/2} \, dt &\leq 
\int_U \L(t)^{1/2} \|u_{Q(t)}(t)\|_\infty^2 \, dt\\
&\lesssim \int_U \L(t)^2 \|u_{Q(t)}(t)\|_2^2 \, dt\\
&< \infty.
\end{split}
\end{equation}
Therefore $f \in L^1(0,T)$ by \eqref{fL}, and hence $u(t)$ is regular on
$(0,T]$ due to \thm{thm:main}.

\end{proof}

\section{Regularity of hyperdissipative NSE}\label{s:hyper}
In this section we will apply the developed technique to give a simple proof of a recent result by T. Tao, \cite{T} on regularity of slightly
supercritical hyperdissipative NSE. Let us consider the system 
\begin{equation} \label{NSE-gen}
\left\{
\begin{aligned}
&\partial_t u + (u \cdot \nabla)u = - \nu D^2 u  - \nabla p = 0, \qquad x \in \mathbb{R}^3, t\geq 0,\\
&\nabla \cdot u =0,\\
&u(0)=u_0,
\end{aligned}
\right.
\end{equation}
where $D$ is a Fourier multiplier whose symbol $m(\xi ) =|\xi|^{5/4}/g(|\xi|)$,
and $g(|\xi|)= \log^{1/4} (2 + |\xi|^2)$.

Consider a Leray-Hopf solution to \eqref{NSE-gen} on $[0,T]$. Due to the
energy inequality we have
\[
\int_0^T \|Du\|_2^2 \, dx < \infty.
\]
Let us fix $\e\in(0,1)$ and use the weak formulation of the NSE with the test-function
$\l_q^{1+\e}(u_q)_q$. As in \cite{CS}, on every interval of regularity of $u(t)$ we obtain
\[
\frac{1}{2}\frac{d}{dt}\sum_{q=-1}^\infty \l_q^{1+\e}\|u_q\|_2^2 + \nu \sum_{q = -1}^\infty \frac{\l_q^{7/2+\e}}{g(\l_q)^2}\|u_q\|_2^2 \leq
C \sum_{q = -1}^\infty \l_q^{3+\e} \|u_q\|_2^2 ( \l_q^{-1} \|u_q\|_\infty),
\]
where $C>0$. Let $c_0=\nu/C$. Now we adapt the definition of $\L$ to these settings as follows:
\[
\L(t)=\min\{\l_q: \l_p^{-3/2} g(\l_p)^2\|u_p(t)\|_\infty \leq c_0\nu \ \forall \ p > q, q\geq 0\}.
\]
This again ensures that the linear term dominates the nonlinear term above $\L(t)$. 
Therefore we have
\[
\frac{1}{2}\frac{d}{dt} E(t)
\lesssim f(t) E(t)
\]
on every interval of regularity of $u(t)$. Here
\[
E(t)=\sum_{q} \l_q^{1+\e}\|u_q\|_2^2, \qquad f(t)=\|u_{\leq Q(t)}(t)\|_{B^{1}_{\infty,\infty}}
\]
as before. Now note that
\[
\begin{split}
f(t) &\leq \frac{1}{c_0\nu} \L^{-3/2}g(\L)^2 \|u_Q\|_\infty \sup_{q \leq Q} \l_q \|u_q\|_\infty\\
&\leq \frac{1}{c_0\nu}g(\L)^2 \L^{5/2} \sup_{q\leq Q} \|u_q\|_2^2\\
&\leq \frac{1}{c_0\nu}g(\L)^4 \|Du\|_2^2,
\end{split}
\]
provided $\L(t) \geq 1$.
In addition, we have
\[
\begin{split}
\log(2+\L^2) &\lesssim \log(2+\L^{1+\epsilon/2})\\
&\leq \log\left(2+\L^{1+\epsilon/2}\frac{1}{c_0\nu}\L^{-3}g(\L)^{4}\|u_Q\|_\infty^2\right)\\
&\lesssim \log\left(2+\frac{1}{c_0\nu}\L^{1+\epsilon}\|u_Q\|_2^2\right)\\
&\leq \log\left(2+\frac{1}{c_0\nu} E(t)\right).
\end{split}
\]
Hence we obtain the following differential inequality:
\[
\frac{1}{2}\frac{d}{dt} E(t)
\lesssim \log\left(2+\frac{1}{c_0\nu} E(t)\right) \|Du(t)\|_2^2 E(t).
\]
Since $\|Du(t)\|_2^2$ is integrable, $E(t)$ does not blow up.

The conjecture made in \cite{T} concerning regularity in the case where $g(|\xi|)= \log^{1/2} (2 + |\xi|^2)$ does not seem to follow from the argument above.

\section{Appendix: Besov-type logarithmic Sobolev inequality}
Although the  logarithmic Sobolev inequality used in the proof of \thm{thm:main} is known (see \cite{jap}) we include its proof here as well for the convenience of the reader. Let $u\in H^s(\R^n)$, for some $s>n/2$. Then 
\begin{equation}\label{log-sob}
\| u\|_{L^\infty} \leq C \|u\|_{B^{0}_{\infty,\infty}} (1 + \log_+ \| u\|_{H^s}).
\end{equation}
Indeed, for any $q\in\N$ we have (up to some absolute constants)
\[
\begin{split}
\|u\|_{L^\infty}& \lesssim \sum_{p\leq q} \|u_p\|_{\infty} + \sum_{p>q} \|u_p\|_{\infty} \\
&\lesssim  q \|u\|_{B^{0}_{\infty,\infty}} + \sum_{p>q} \l_p^{n/2-s} \l_p^s \|u_p\|_{2} \\
&\lesssim q \|u\|_{B^{0}_{\infty,\infty}} + \l_q^{n/2 - s} \| u\|_{H^s}.
\end{split}
\]
Optimizing over $q$ immediately gives \eqref{log-sob}.


\end{document}